\documentclass[12pt,oneside,british]{amsart}
\usepackage[T1]{fontenc}
\usepackage[latin9]{luainputenc}
\usepackage{geometry}
\usepackage{babel}
\usepackage{float}
\usepackage{mathrsfs}
\usepackage{amstext}
\usepackage{amsthm}
\usepackage{amssymb}
\makeindex
\PassOptionsToPackage{normalem}{ulem}
\usepackage{ulem}
\usepackage[unicode=true,pdfusetitle,
 bookmarks=true,bookmarksnumbered=false,bookmarksopen=false,
 breaklinks=false,pdfborder={0 0 1},backref=false,colorlinks=false]
 {hyperref}

\makeatletter

\providecommand{\tabularnewline}{\\}

\numberwithin{equation}{section}
\numberwithin{figure}{section}
\theoremstyle{plain}
\newtheorem{thm}{\protect\theoremname}[section]
\theoremstyle{definition}
\newtheorem{defn}[thm]{\protect\definitionname}
\theoremstyle{remark}
\newtheorem{rem}[thm]{\protect\remarkname}
\theoremstyle{plain}
\newtheorem{prop}[thm]{\protect\propositionname}

\makeatother

\providecommand{\definitionname}{Definition}
\providecommand{\propositionname}{Proposition}
\providecommand{\remarkname}{Remark}
\providecommand{\theoremname}{Theorem}

\begin{document}
\title{$(\Theta_{n},sl_{n})$ - graded Lie algebras $(n=3,4)$}
\author{Hogir Mohammed Yaseen}
\address{Department of Mathematics, Salahaddin University-Erbil, Iraq.}
\email{hogr.yaseen@su.edu.krd, hogrmyaseen@gmail.com}
\begin{abstract}
Let $\mathbb{F}$ be a field of characteristic zero and let $\mathfrak{g}$
be a non-zero finite-dimensional split semisimple Lie algebra with
root system $\Delta$. Let $\Gamma$ be a finite set of integral weights
of $\mathfrak{g}$ containing $\Delta$ and $\{0\}$. Following \cite{key-2,key-9},
we say that a Lie algebra $L$ over $\mathbb{F}$ is \emph{generalized
root graded}, or more exactly $(\Gamma,\mathfrak{g})$-\emph{graded},
if $L$ contains a semisimple subalgebra isomorphic to $\mathfrak{g}$,
the $\mathfrak{g}$-module $L$ is the direct sum of its weight subspaces
$L_{\alpha}$ ($\alpha\in\Gamma$) and $L$ is generated by all $L_{\alpha}$
with $\alpha\ne0$ as a Lie algebra. Let $\mathfrak{g}\cong sl_{n}$
and 
\[
\Theta_{n}=\{0,\pm\varepsilon_{i}\pm\varepsilon_{j},\pm\varepsilon_{i},\pm2\varepsilon_{i}\mid1\leq i\neq j\leq n\}
\]
where $\{\varepsilon_{1},\dots,\varepsilon_{n}\}$ is the set of weights
of the natural $sl_{n}$-module. In \cite{key-10}, we classify $(\Theta_{n},sl_{n})$-graded
Lie algebras for $n>4$. In this paper we describe the multiplicative
structures and the coordinate algebras of $(\Theta_{n},sl_{n})$-graded
Lie algebras $(n=3,4)$. In $n=3$, we assume that 
\[
[V(2\omega_{1})\otimes C,V(2\omega_{1})\otimes C]=[V(2\omega_{2})\otimes C',V(2\omega_{2})\otimes C']=0
\]
 where $V(\omega)$ is the simple $\mathfrak{g}$-module of highest
weight $\omega$, $C={\rm Hom_{\mathfrak{g}}}(V(2\omega_{1}),L)$
and  $C'={\rm Hom_{\mathfrak{g}}}(V(2\omega_{2}),L)$ .
\end{abstract}

\maketitle
\tableofcontents{}

\allowdisplaybreaks

\global\long\def\bbF{\mathbb{F}}%

\global\long\def\ffg{\mathfrak{g}}%

\global\long\def\ffa{\mathfrak{a}}%

\global\long\def\ffb{\mathfrak{b}}%

\global\long\def\Ker{\operatorname{\rm Ker}}%

\global\long\def\ad{\operatorname{\rm ad}}%

\global\long\def\Der{\operatorname{\rm Der}}%

\global\long\def\Rad{\operatorname{\rm Rad}}%

\global\long\def\tra{\operatorname{\rm tr}}%

\global\long\def\hfb{\operatorname{\rm HF}}%

\global\long\def\Hom{\operatorname{\rm Hom}}%

\global\long\def\End{\operatorname{\rm End}}%

\global\long\def\hom{\operatorname{\rm Hom_{\mathfrak{g}}}}%

\global\long\def\dim{\operatorname{\rm dim}}%

\global\long\def\ki{\operatorname{\rm B}}%

\section{\label{sec:Multiplication-in--graded}Introduction}

In 1992, Berman and Moody introduced root-graded Lie algebras to study
toroidal Lie algebras and Slodowy matrix algebras intersection. Nonetheless,
this concept previously appeared in the study of simple Lie algebras
by Seligman \cite{key-10}. Root graded Lie algebras of simply-laced
finite root systems were classified up to central isogeny by Berman
and Moody in \cite{key-7}. The case of double-laced finite root systems
was settled by Benkart and Zelmanov \cite{key-6}. Non-reduced systems
$BC_{n}$ were considered by Allison, Benkart and Gao \cite{key-1}
(for $n\text{\ensuremath{\ge}}2$) and by Benkart and Smirnov \cite{key-5}
(for $n=1$). The aim of this paper is to describe the multiplicative
structures and the coordinate algebras of $(\Theta_{n},sl_{n})$-graded
Lie algebras where $n=3,4$ and 
\begin{eqnarray*}
\Theta_{n} & = & \{0,\pm\varepsilon_{i}\pm\varepsilon_{j},\pm\varepsilon_{i},\pm2\varepsilon_{i}\mid1\leq i\neq j\leq n\}
\end{eqnarray*}

and $\{\varepsilon_{1},\dots,\varepsilon_{n}\}$ is the set of weights
of the natural $sl_{n}$-module. Throughout the paper, the ground
field $\bbF$ is of characteristic zero and the grading subalgebra
$\mathfrak{g}\cong sl_{n}$. We denote by $\Theta_{n}^{+}$ the set
of dominant weights in $\Theta_{n}$ and the corresponding simple
$sl_{n}$-modules. Thus,
\begin{align*}
\Theta_{4}^{+} & =\{\omega_{1}+\omega_{3}=\varepsilon_{1}-\varepsilon_{4},\,\omega_{1}=\varepsilon_{1},\,\omega_{3}=-\varepsilon_{4},2\omega_{1}=2\varepsilon_{1},\ 2\omega_{3}=-2\varepsilon_{4},\,\omega_{2}=\varepsilon_{1}+\varepsilon_{2},\,0\},\\
\Theta_{3}^{+} & =\{\omega_{1}+\omega_{3}=\varepsilon_{1}-\varepsilon_{4},\,\omega_{1}=\varepsilon_{1},\,\omega_{3}=-\varepsilon_{4},2\omega_{1}=2\varepsilon_{1},\ 2\omega_{3}=-2\varepsilon_{4},\,0\}.
\end{align*}
These are the highest weights of the irreducible $\mathfrak{g}$-
modules. First we compute tensor product decompositions for the modules
in $\Theta_{n}^{+}$ . Then we describe the multiplicative structures
of $(\Theta_{n},sl_{n})$ - graded Lie algebras. The coordinate algebra
of these Lie algebras are described in Section \ref{4}.

\section{\label{2}tensor product decompositions for the modules in $\Theta_{n}^{+}$
$(n=3,4)$.}

We begin with the general definition of Lie algebras graded by finite
weight systems.
\begin{defn}
\cite{key-2} \label{def of gamma} Let $\Delta$ be a root system
and let $\Gamma$ be a finite set of integral weights of $\Delta$
containing $\Delta$ and $\{0\}$. A Lie algebra $L$ is called $(\Gamma,\mathfrak{g})$-\emph{graded
}(or simply\emph{ $\Gamma$-graded}) if 

$(\Gamma1)$ $L$ contains as a subalgebra a non-zero finite-dimensional
split semisimple Lie algebra 
\[
\mathit{\mathrm{\mathfrak{g}=\mathfrak{h}}\oplus\mathrm{\underset{\alpha\in\mathit{\text{\ensuremath{\Delta}}}}{\bigoplus}}}\mathit{\mathrm{\mathfrak{g}}}_{\alpha},
\]
 whose root system is $\Delta$ relative to a split Cartan subalgebra
$\text{\ensuremath{\mathfrak{h}}}=\mathfrak{g}_{0}$;

$(\Gamma2)$ $L=\underset{\alpha\in\Gamma}{\bigoplus}L_{\alpha}$
where $L_{\alpha}=\left\{ x\in\mathrm{\mathit{L}\mid\left[\mathit{h,x}\right]=\mathit{\alpha\left(h\right)x\text{ for all }h\in\mathrm{\mathit{\mathfrak{h}}}}}\right\} $;

$(\Gamma3)$ $L_{0}=\underset{\alpha,-\alpha\in\Gamma\setminus\{0\}}{\overset{}{\sum}}\left[L_{\alpha},L_{-\alpha}\right]$.
\end{defn}

The subalgebra $\mathfrak{g}$ is called the \emph{grading subalgebra}
of $L$. If $\mathscr{\mathfrak{g}}$ is the split simple Lie algebra
and $\Gamma=\Delta\cup\{0\}$ then $L$ is said to be \emph{root-graded}.

We fix a base $\Pi=\{\varepsilon_{i}-\varepsilon_{i+1}\text{ for }i=1,2,\cdots,n-1\}$
of simple roots for the root system $A_{n-1}=\{\varepsilon_{i}-\varepsilon_{j}\mid1\leq i\neq j\leq n\}$.
Let $L$ be a Lie algebra containing a non-zero split simple subalgebra
$\mathfrak{g}$. Using the same arguments of \cite[Lemma 3.1.2 and Proposition 3.2.2]{key-9}
we get the following:

(1) A Lie algebra $L$ is $(\Theta_{3},\mathfrak{g})$-graded if and
only if $L$ is generated by $\mathfrak{g}$ as an ideal and the $\mathfrak{g}$-module
$L$ decomposes into copies of the adjoint module, its symmetric $S^{2}V$,
its exterior squares $\wedge^{2}V$, their duals and the one dimensional
trivial $\mathfrak{g}$-module.

(2) A Lie algebra $L$ is $(\Theta_{4},\mathfrak{g})$-graded if and
only if $L$ is generated by $\mathfrak{g}$ as an ideal and the $\mathfrak{g}$-module
$L$ decomposes into copies of the adjoint module (we will denote
it by the same letter $\mathfrak{g}$), the natural module $V$, its
exterior squares $\wedge^{2}V$, its symmetric $S^{2}V$, their duals
and the one dimensional trivial $\mathfrak{g}$-module.

Thus, by collecting isotypic components, we get the following decomposition
of the $\mathfrak{g}$-module $L$:
\begin{align}
L & =(\mathfrak{g}\otimes A)\oplus(S\otimes C)\oplus(S'\otimes C')\oplus(\Lambda\otimes E)\oplus(\Lambda'\otimes E')\oplus D\text{ for }n=3;\label{eq:drezh mine}\\
L & =(\mathfrak{g}\otimes A)\oplus(V\otimes B)\oplus(V'\otimes B')\oplus(S\otimes C)\oplus(S'\otimes C')\oplus(\Lambda\otimes E)\oplus D\;\;\text{ for }n=4;\nonumber 
\end{align}
where $A,B,B',C,C',E$ are vector spaces, 
\begin{align*}
 & \mathfrak{g}:=V(\omega_{1}+\omega_{n-1}),\quad V:=V(\omega_{1}),\quad V':=V(\omega_{n-1}),\\
 & S:=V(2\omega_{1}),\quad S':=V(2\omega_{n-1}),\quad\Lambda:=V(\omega_{2}),\quad\Lambda':=V(\omega_{n-2})
\end{align*}
 and $D$ is the sum of the trivial $\mathfrak{g}$-modules. Note
that $\Lambda\cong\Lambda'$ for $n=4$.

Alternatively, these spaces can also be viewed as the corresponding
$\mathfrak{g}$-mod Hom-spaces: $A=\hom(\mathfrak{g},L)$, $B=\hom(V,L)$,
etc, so for each simple $\mathfrak{g}$-module $M$, the space $M\otimes\hom(M,L)$
is canonically identified with the $M$-isotypic component of $L$
via the evaluation map 
\begin{equation}
M\otimes\hom(M,L)\rightharpoondown L,\,\;m\otimes\varphi\mapsto\varphi(m),\label{eq:MHom}
\end{equation}

see \cite[Proposition 4.1.15]{key-8}. 

We can write the following examples of $\Theta_{n}$-graded Lie algebras
$(n=3,4)$:

1) Let $L=sl_{n+k}$ and let $\mathfrak{g}$ be the copy of $sl_{n}$
in the northwest corner. We consider the adjoint action of $\mathfrak{g}$
on $L$. Then the $\mathfrak{g}$-module $L$ decomposes into $k$
copies of the natural module $V=\mathbb{F}^{n}$, $k$ copies of the
dual module $V'=\Hom(V,\mathbb{F})$, an adjoint module $\mathfrak{g}$
and one dimensional trivial $\mathfrak{g}$-modules in its southeast
corner. Then

\[
L=\mathfrak{g}\oplus V^{\oplus k}\oplus V'^{\oplus k}\oplus D
\]
where $D$ is the sum of the trivial $sl_{n}$-modules. As a result,
we may write 
\[
L=\mathfrak{g}\oplus(V\otimes B)\oplus(V'\otimes B')\oplus D
\]
where $B\cong B'\cong\mathbb{F}^{k}$. Then $L$ is $(A_{n-1},\mathfrak{g})$-graded.
Bahturin and Benkart \cite{key-3} (for $n>3$) and Benkart and Elduque
\cite{key-4} (for $n=3$) described the multiplicative structure
of this type of Lie algebras.

2) Any Lie algebra which is $(A_{n-1},sl_{n})$-graded is also $\Theta_{n}$-graded.
For such a Lie algebra, the space $L_{\alpha}=\{0\}$ for all $\alpha$
not in $A_{n-1}$.

3) Let $L=sl_{2n+1}$ and $\mathfrak{g}=\left\{ \left[\begin{array}{ccc}
x & 0 & 0\\
0 & -x^{t} & 0\\
0 & 0 & 0
\end{array}\right]\mid x\in sl_{n}\right\} \subset L$ where $n=3,4$. We consider the adjoint action of $\mathfrak{g}$
on $L$. Then $L$ is $(\Theta_{n},\mathfrak{g})$-graded .
\begin{defn}
\label{Lem-Vdual} (1) We identify the $\mathfrak{g}$-modules $V$
and $V'$ with the space $\mathbb{F}^{n}$$(n=3,4)$ of column vectors
with the following actions: 
\begin{align*}
 & x.v=xv\text{ \text{ for }}x\in sl_{n},\,v\in V,\\
 & x.v'=-x^{t}v'\text{ for }x\in sl_{n},\,v'\in V'.
\end{align*}
(2) We identify $S$ and $S'$ (resp. $\Lambda$ and $\Lambda'$)
with symmetric (resp. skew-symmetric) $n\times n$ matrices over $\mathbb{F}$
$(n=3,4)$. Then $S$, $S'$, $\Lambda$ and $\Lambda'$ are $\mathfrak{g}$-modules
under the actions: 
\begin{align*}
 & x.s=xs+sx^{t}\text{ for }x\in sl_{n},\,s\in S,\\
 & x.\lambda=x\lambda+\lambda x^{t}\text{ for }x\in sl_{n},\,\lambda\in\Lambda,\\
 & x.s'=-s'x-x^{t}s'\text{ for }x\in sl_{n},\,s'\in S,\\
 & x.\lambda'=-\lambda'x-x^{t}\lambda'\text{ for }x\in sl_{n},\,\lambda'\in\Lambda'.
\end{align*}
\end{defn}

Since the subalgebra $\mathfrak{g}$ of $L$ is a $\mathfrak{g}$-submodule,
there exists a distinguished element $1$ of $A$ such that $\mathfrak{g}=\mathfrak{g}\otimes1$.
In particular, 
\begin{equation}
[x\otimes1,y\otimes b]=x.y\otimes b.\label{eq:iden 1}
\end{equation}
where $x\otimes1$ is in $\mathfrak{g}\otimes1$, $y\otimes b$ belongs
to one of the components in (\ref{eq:drezh mine}), and $x.y$ is
as in Definition \ref{Lem-Vdual}.

From equation \ref{eq:drezh mine} and Definition \ref{Lem-Vdual}
and the properties that $V\cong\Lambda'$ and $V'\cong\Lambda$ for
$sl_{3}$, we get the following decomposition of the $\mathfrak{g}$-module
$L$:
\begin{equation}
L=(\mathfrak{g}\otimes A)\oplus(S\otimes C)\oplus(S'\otimes C')\oplus(\Lambda\otimes E)\oplus(\Lambda'\otimes E')\oplus D\label{eq:sl3 banotation}
\end{equation}

From equation \ref{eq:drezh mine} and Definition \ref{Lem-Vdual}
and the property that $\Lambda\cong\Lambda'$ for $sl_{4}$, we get
the following decomposition of the $\mathfrak{g}$-module $L$:
\begin{equation}
L=(\mathfrak{g}\otimes A)\oplus(V\otimes B)\oplus(V'\otimes B')\oplus(S\otimes C)\oplus(S'\otimes C')\oplus\oplus(\Lambda\otimes E)\oplus D.\label{eq:sl4 ba notation}
\end{equation}

Using \cite[Section 3.4]{key-9}, in Tables \ref{t12-2}-\ref{t13}
we get the following $\Theta$-components of all tensor product decompositions
for the modules in $\Theta_{n}^{+}$ $(n=3,4)$. If the cell in row
$X$ and column $Y$ contains $Z$ this means that $\Theta(X\otimes Y)\cong Z$.

\begin{table}[H]
\begin{tabular}{|c|r@{\extracolsep{0pt}.}l|c|c|c|c|}
\hline 
$\otimes$ & \multicolumn{2}{c|}{$\mathfrak{g}$} & $S$ & $S'$ & $V\cong\Lambda'$ & $V'\cong\Lambda$\tabularnewline
\hline 
\hline 
$\mathfrak{g}$ & \multicolumn{2}{c|}{$\mathfrak{g+\mathfrak{g}}+T$} & $S+\Lambda$ & $S'+\Lambda'$ & $S'+\Lambda'$ & $S+\Lambda$\tabularnewline
\hline 
$S$ & \multicolumn{2}{c|}{$S+\Lambda$} & $S'$ & $\mathfrak{g}+T$ & $\mathfrak{g}$ & $\Lambda'$\tabularnewline
\hline 
$S'$ & \multicolumn{2}{c|}{$S'+\Lambda'$} & $\mathfrak{g}+T$ & $S$ & $\Lambda$ & $\mathfrak{g}$\tabularnewline
\hline 
$V$ & \multicolumn{2}{c|}{$S'+\Lambda'$} & $\mathfrak{g}$ & $\Lambda$ & $S+\Lambda$ & $\mathfrak{g}+T$\tabularnewline
\hline 
$V'$ & \multicolumn{2}{c|}{$S+\Lambda$} & $\Lambda'$ & $\mathfrak{g}$ & $\mathfrak{g}+T$ & $S'+\Lambda'$\tabularnewline
\hline 
\end{tabular}

\caption{$\Theta$-component of tensor product decompositions for $sl_{3}$}

\label{t12-2}
\end{table}

\begin{table}[H]
\begin{tabular}{|c|r@{\extracolsep{0pt}.}l|c|c|c|c|c|}
\hline 
$\otimes$ & \multicolumn{2}{c|}{$\mathfrak{g}$} & $S$ & $\Lambda\cong\Lambda'$ & $S'$ & $V$ & $V'$\tabularnewline
\hline 
\hline 
$\mathfrak{g}$ & \multicolumn{2}{c|}{$\mathfrak{g+\mathfrak{g}}+T$} & $S+\Lambda$ & $\begin{array}{c}
S+\Lambda+S'\end{array}$ & $S'+\Lambda$ & $V$ & $V'$\tabularnewline
\hline 
$S$ & \multicolumn{2}{c|}{$S+\Lambda$} & $0$ & $\mathfrak{g}$ & $\mathfrak{g}+T$ & $0$ & $V$\tabularnewline
\hline 
$\Lambda$ & \multicolumn{2}{c|}{$S+\Lambda+S'$} & $\mathfrak{g}$ & $\mathfrak{g}+T$ & $\mathfrak{g}$ & $\begin{array}{c}
V'\end{array}$ & $V$\tabularnewline
\hline 
$S'$ & \multicolumn{2}{c|}{$S'+\Lambda$} & $\mathfrak{g}+T$ & $\mathfrak{g}$ & 0 & $V'$ & $0$\tabularnewline
\hline 
$V$ & \multicolumn{2}{c|}{$V$} & $0$ & $\begin{array}{c}
V'\end{array}$ & $V'$ & $S+\Lambda$ & $\mathfrak{g}+T$\tabularnewline
\hline 
$V'$ & \multicolumn{2}{c|}{$V'$} & $V$ & $V$ & $0$ & $\mathfrak{g}+T$ & $S'+\Lambda$\tabularnewline
\hline 
\end{tabular}

\caption{$\Theta$-component of tensor product decompositions for $sl_{4}$}

\label{t13}
\end{table}

In Tables \label{t3-1-1} and \ref{t3} below , if the cell in row
$X$ and column $Y$ contains $Z$, this means that there is a bilinear
map $X\otimes Y\rightarrow Z$ given by $x\otimes y\mapsto(x,y)_{Z}$.
For simplicity of notation, we will write $dy$ instead of $(d,y)_{D}$
if $X=Z=D$ and we will write $\langle x,y\rangle$ instead of $(x,y)_{D}$
if $X,Y\neq D$ and $Z=D.$ In the case $X=Y=Z=A$, we have two bilinear
products $a_{1}\otimes a_{2}\mapsto a_{1}\circ a_{2}$ and $a_{1}\otimes a_{2}\mapsto[a_{1},a_{2}]$
for $a_{1},a_{2}\in A$. Note that some of the cells are empty. The
corresponding products $X\otimes Y\rightarrow Z$ will be defined
later by extending the existing maps $Y\otimes X\rightarrow Z$. This
will make the table symmetric. 

\begin{table}[H]
\begin{tabular}{|c|r@{\extracolsep{0pt}.}l|c|c|c|c|c|}
\hline 
$.$ & \multicolumn{2}{c|}{$A$} & $C$ & $C'$ & $E$ & $E'$ & $D$\tabularnewline
\hline 
\hline 
$A$ & \multicolumn{2}{c|}{$\begin{array}{c}
(A,\circ,[\ ]),\,D\end{array}$} & $C,E$ &  & $C,E$ &  & \tabularnewline
\hline 
$C$ & \multicolumn{2}{c|}{} & $0$ & $A,D$ & $E'$ & $A$ & \tabularnewline
\hline 
$C'$ & \multicolumn{2}{c|}{$C',E'$} &  & $0$ & $A$ & $E$ & \tabularnewline
\hline 
$E$ & \multicolumn{2}{c|}{} & $E'$ &  & $C',E'$ & $A,D$ & \tabularnewline
\hline 
$E'$ & \multicolumn{2}{c|}{$C',E'$} &  & $E$ &  & $C,E$ & \tabularnewline
\hline 
$D$ & \multicolumn{2}{c|}{$A$} & $C$ & $C'$ & $E$ & $E'$ & $D$\tabularnewline
\hline 
\end{tabular}

\caption{Bilinear products for $n=3$}

\label{t3-1}
\end{table}

\begin{table}[H]
\begin{tabular}{|c|r@{\extracolsep{0pt}.}l|c|c|c|c|c|c|}
\hline 
$.$ & \multicolumn{2}{c|}{$A$} & $B$ & $B'$ & $C$ & $C'$ & $E$ & $D$\tabularnewline
\hline 
\hline 
$A$ & \multicolumn{2}{c|}{$\begin{array}{c}
(A,\circ,[\ ]),\,D\end{array}$} & $B$ &  & $C,E$ &  & $C,E,C'$ & \tabularnewline
\hline 
$B$ & \multicolumn{2}{c|}{} & $C,E$ & $A,D$ & $0$ &  & $0$ & \tabularnewline
\hline 
$B'$ & \multicolumn{2}{c|}{$A$} &  & $C',E$ & $B$ & $0$ & $B$ & \tabularnewline
\hline 
$C$ & \multicolumn{2}{c|}{} & $0$ &  & $0$ & $A,D$ & $0$ & \tabularnewline
\hline 
$C'$ & \multicolumn{2}{c|}{$C',E$} & $B'$ & $0$ &  & $0$ & $A$ & \tabularnewline
\hline 
$E$ & \multicolumn{2}{c|}{} & $0$ &  & $0$ &  & $0$ & \tabularnewline
\hline 
$D$ & \multicolumn{2}{c|}{$A$} & $B$ & $B'$ & $C$ & $C'$ & $E$ & $D$\tabularnewline
\hline 
\end{tabular}

\caption{Bilinear products for $n=4$}

\label{t3}
\end{table}

.Let $x$ and $y$ be $n\times n$ matrices. We will use the following
products:
\begin{align*}
[x,y] & =xy-yx,\\
x\circ y & =xy+yx-\frac{2}{n}\tra(xy)I,\\
x\diamond y & =xy+yx,\\
(x\mid y) & =\ensuremath{\frac{1}{n}\tra(xy)}.
\end{align*}

\section{\label{3}Multiplication in $(\Theta_{n},sl_{n})$-graded Lie algebras
$(n=3,4)$}

\subsection{Multiplication in $\Theta_{3}$-graded Lie algebras}

Let $L$ be a $\Theta_{3}$-graded Lie algebra with the grading subalgebra
$\mathfrak{g}\cong sl_{3}$ and the properties that 
\begin{equation}
[V(2\omega_{1})\otimes C,V(2\omega_{1})\otimes C]=[V(2\omega_{2})\otimes C',V(2\omega_{2})\otimes C']=0.\label{condition}
\end{equation}
 where $V(\omega)$ is the simple $\mathfrak{g}$-module of highest
weight $\omega$, $C={\rm Hom_{\mathfrak{g}}}(V(2\omega_{1}),L)$
and  $C'={\rm Hom_{\mathfrak{g}}}(V(2\omega_{2}),L)$ . In \ref{eq:sl3 banotation},
we show that $L$ is the direct sum of finite-dimensional irreducible
$\mathfrak{g}$-modules whose highest weights are in $\Theta_{3}^{+}$,
i.e. as a $\mathfrak{g}$-module, 
\[
L=(\mathfrak{g}\otimes A)\oplus(S\otimes C)\oplus(S'\otimes C')\oplus(\Lambda\otimes E)\oplus(\Lambda'\otimes E')\oplus D.
\]
Note that $V'=U(sl_{3})e_{3}$ and $\Lambda=U(sl_{3})(E_{1,2}-E_{2,1})$
are highest weight modules with highest weight $\omega_{2}$ where
$E_{i,j}$ denote the matrix units and $e_{3}=\left[\begin{array}{c}
0\\
0\\
1
\end{array}\right]$ (resp. $V=U(sl_{4})e_{1}$ and $\Lambda'=U(sl_{4})(E_{3,4}-E_{4,3})$
are highest weight module with highest weight $\omega_{1}$ where
$e_{1}=\left[\begin{array}{c}
1\\
0\\
0
\end{array}\right]$). Define $f:\Lambda'\rightarrow V$ by $f(x.(E_{3,4}-E_{4,3})=x.e_{1}$
and $g:\Lambda\rightarrow V'$ by $g(x.(E_{1,2}-E_{2,1}))=x.e_{3}$
for all $x\in U(sl_{4})$. This allows us to identify $\Lambda$ with
$V'$ and $\Lambda'$ with $V$.

In (\ref{t2-1}) we list bases for all non-zero $\mathfrak{g}$-module
homomorphism spaces $\hom(X\otimes Y,Z)$ where $X,Y,Z\in\{\mathfrak{g},V,V',S,S',T\}$
and $X$ and $Y$ are both non-trivial. Note that all of them are
$1$-dimensional except the first one (which is 2-dimensional).

\begin{align}
\hom(\mathfrak{g}\otimes\mathfrak{g},\mathfrak{g}) & =span\{x\otimes y\mapsto xy-yx,\ x\otimes y\mapsto xy+yx-\frac{2}{3}\tra(xy)I\},\label{t2-1}\\
\hom(\Lambda\otimes\Lambda',\mathfrak{g}) & =span\{\lambda\otimes\lambda'\mapsto\lambda\lambda'-\frac{\tra(\lambda\lambda')}{3}I\},\nonumber \\
\hom(\Lambda\otimes\Lambda,\Lambda') & =span\{\lambda_{1}\otimes\lambda_{2}\mapsto f(\lambda'_{1})(f(\lambda'_{2}))^{t}-f(\lambda'_{2})(f(\lambda'_{1}))^{t}\},\nonumber \\
\hom(\Lambda\otimes\Lambda,S') & =span\{\lambda_{1}\otimes\lambda_{2}\mapsto f(\lambda'_{1})(f(\lambda'_{2}))^{t}+f(\lambda'_{2})(f(\lambda'_{1}))^{t}\},\nonumber \\
\hom(\Lambda\otimes\Lambda,\Lambda) & =span\{\lambda_{1}\otimes\lambda_{2}\mapsto g(\lambda_{1})(g(\lambda_{2}))^{t}-g(\lambda_{2})(g(\lambda_{1}))^{t}\},\nonumber \\
\hom(\Lambda\otimes\Lambda,S) & =span\{\lambda_{1}\otimes\lambda_{2}\mapsto g(\lambda_{1})(g(\lambda_{2}))^{t}+g(\lambda_{2})(g(\lambda_{1}))^{t}\},\nonumber \\
\hom(S\otimes\Lambda',\mathfrak{g}) & =span\{s\otimes\lambda'\mapsto s\lambda'\},\nonumber \\
\hom(S'\otimes\Lambda,\mathfrak{g}) & =span\{s'\otimes\lambda\mapsto s'\lambda\},\nonumber \\
\hom(S\otimes S',\mathfrak{g}) & =span\{s\otimes s'\mapsto ss'-\frac{\tra(ss')}{3}I\},\nonumber \\
\hom(\Lambda'\mathfrak{\otimes g},\Lambda') & =span\{\lambda'\otimes x\mapsto\lambda'x+x^{t}\lambda'\},\nonumber \\
\hom(S\otimes\Lambda,\Lambda') & =span\{s\otimes\lambda\mapsto sg(\lambda)\},\nonumber \\
\hom(\mathfrak{g}\otimes\Lambda,S) & =span\{x\otimes\lambda\mapsto x\lambda-\lambda x^{t}\},\nonumber \\
\hom(\mathfrak{g}\otimes\Lambda,\Lambda) & =span\{x\otimes\lambda\mapsto x\lambda+\lambda x^{t}\},\nonumber \\
\hom(S'\otimes\Lambda',\Lambda) & =span\{s'\otimes\lambda'\mapsto sf(\lambda')\},\nonumber \\
\hom(\mathfrak{g}\otimes S,S) & =span\{x\otimes s\mapsto xs+sx^{t}\},\nonumber \\
\hom(S'\mathfrak{\otimes g},S') & =span\{s'\otimes x\mapsto s'x+x^{t}s'\},\nonumber \\
\hom(\Lambda'\mathfrak{\otimes g},S') & =span\{\lambda'\otimes x\mapsto\lambda'x-x^{t}\lambda'\},\nonumber \\
\hom(\mathfrak{g}\otimes S,\Lambda) & =span\{x\otimes s\mapsto xs-sx^{t}\},\nonumber \\
\hom(S'\mathfrak{\otimes g},\Lambda') & =span\{s'\otimes x\mapsto s'x-x^{t}s'\},\nonumber \\
\hom(\mathfrak{g}\otimes\mathfrak{g},T) & =span\{x_{1}\otimes x_{2}\mapsto\frac{1}{3}\tra(x_{1}x_{2})\},\nonumber \\
\hom(\Lambda\otimes\Lambda',T) & =span\{\lambda\otimes\lambda'\mapsto\frac{1}{3}\tra(\lambda\lambda')\}.\nonumber \\
\hom(S\otimes S',T) & =span\{s\otimes s'\mapsto\frac{1}{3}\tra(ss')\},\nonumber 
\end{align}

\label{mutiplication-1} Following the methods in \cite{key-1,key-10,key-7,key-6},
using the results of (\ref{t2-1}) and Table\ref{t12-2} , we may
suppose that the multiplication in $L$ is given as follows. For all
$x,y\in sl_{3}$, $s\in S$, $\lambda_{1},\lambda_{2}\in\Lambda$,
$s'\in S'$, $\lambda'_{1},\lambda'_{2}\in\Lambda'$ and for all $a,a_{1},a_{2}\in A$,
$c\in C$, $c'\in C'$, $e\in E$, $e'\in E'$ and $d,d_{1},d_{2}\in D$,

\begin{flalign}
 & [x\otimes a_{1},y\otimes a_{2}]=(x\circ y)\otimes\frac{[a_{1},a_{2}]}{2}+[x,y]\otimes\frac{a_{1}\circ a_{2}}{2}+(x\mid y)\langle a_{1},a_{2}\rangle,\label{main for-2}\\
 & [\lambda\otimes e,\lambda'\otimes e']=(\lambda\lambda'-(\lambda\mid\lambda')I)\otimes(e,e')_{A}+(\lambda\mid\lambda')\langle e,e'\rangle=-[\lambda'\otimes e',\lambda\otimes e],\nonumber \\
 & [\lambda_{1}\otimes e_{1},\lambda_{2}\otimes e_{2}]=(g(\lambda_{1})(g(\lambda_{2}))^{t}+g(\lambda_{2})(g(\lambda_{1}))^{t})\otimes\frac{(e_{1},e_{2})_{C}}{2}\nonumber \\
 & \,\,\,\,\,\,\,\,\,\,\,\,\,\,\,\,\,\,\,\,+(g(\lambda_{1})(g(\lambda_{2}))^{t}-g(\lambda_{2})(g(\lambda_{1}))^{t})\otimes\frac{(e_{1},e_{2})_{E}}{2},\nonumber \\
 & [\lambda'_{1}\otimes e'_{1},\lambda'_{2}\otimes e'_{2}]=(f(\lambda'_{1})(f(\lambda'_{2}))^{t}+f(\lambda'_{2})(f(\lambda'_{1}))^{t})\otimes\frac{(e'_{1},e'_{2})_{C'}}{2}\nonumber \\
 & \,\,\,\,\,\,\,\,\,\,\,\,\,\,\,\,\,\,\,\,\,+(f(\lambda'_{1})(f(\lambda'_{2}))^{t}-f(\lambda'_{2})(f(\lambda'_{1}))^{t})\otimes\frac{(e'_{1},e'_{2})_{E'}}{2},\nonumber \\
 & [s\otimes c,s'\otimes c']=(ss'-(s\mid s')I)\otimes(c,c')_{A}+(s\mid s')\langle c,c'\rangle=-[s'\otimes c',s\otimes c],\nonumber \\
 & [x\otimes a,s\otimes c]=(xs+sx^{t})\otimes\frac{(a,c)_{C}}{2}+(xs-sx^{t})\otimes\frac{(a,c)_{E}}{2}=-[s\otimes c,x\otimes a],\nonumber \\
 & [x\otimes a,\lambda\otimes e]=(x\lambda+\lambda x^{t})\otimes\frac{(a,e)_{E}}{2}+(x\lambda-\lambda x^{t})\otimes\frac{(a,e)_{C}}{2}=-[\lambda\otimes e,x\otimes a],\nonumber \\
 & [s'\otimes c',x\otimes a]=(s'x+x^{t}s')\otimes\frac{(c',a)_{C'}}{2}+(s'x-x^{t}s')\otimes\frac{(c',a)_{E'}}{2}=-[x\otimes a,s'\otimes c'],\nonumber \\
 & [\lambda'\otimes e',x\otimes a]=(\lambda'x+x^{t}\lambda')\otimes\frac{(e',a)_{E'}}{2}+(\lambda'x-x^{t}\lambda')\otimes\frac{(e',a)_{C'}}{2}=-[x\otimes a,\lambda'\otimes e'],\nonumber \\
 & [s\otimes c,\lambda'\otimes e']=s\lambda'\otimes(c,e')_{A}=-[\lambda'\otimes e',s\otimes c],\nonumber \\
 & [s'\otimes c',\lambda\otimes e]=s'\lambda\otimes(c',e)_{A}=-[\lambda\otimes e,s'\otimes c'],\nonumber \\
 & [s'\otimes c',\lambda'\otimes e']=s'f(\lambda')\otimes(c',e')_{E}=-[\lambda'\otimes e',s'\otimes c'],\nonumber \\
 & [\lambda\otimes e,s\otimes c]=sg(\lambda)\otimes(e,c)_{E'}=-[s\otimes c,\lambda\otimes e],\nonumber \\
 & [d,x\otimes a]=x\otimes da=-[x\otimes a,d],\nonumber \\
 & [d,s\otimes c]=s\otimes dc=-[s\otimes c,d],\nonumber \\
 & [d,\lambda\otimes e]=\lambda\otimes de=-[\lambda\otimes e,d],\nonumber \\
 & [d,s'\otimes c']=s'\otimes dc'=-[s'\otimes c',d],\nonumber \\
 & [d,\lambda'\otimes e']=\lambda'\otimes de'=-[\lambda'\otimes e',d],\nonumber \\
 & [d_{1},d_{2}]\in D,\nonumber 
\end{flalign}

\subsection{Multiplication in $\Theta_{4}$-graded Lie algebras}

Let $L$ be a $\Theta_{4}$-graded Lie algebra with the grading subalgebra
$\mathfrak{g}\cong sl_{4}$. In \ref{eq:sl3 banotation}, we show
that $L$ is the direct sum of finite-dimensional irreducible $\mathfrak{g}$-modules
whose highest weights are in $\Theta_{3}^{+}$, i.e. as a $\mathfrak{g}$-module,
\[
L=(\mathfrak{g}\otimes A)\oplus(V\otimes B)\oplus(V'\otimes B')\oplus(S\otimes C)\oplus(S'\otimes C')\oplus(\Lambda\otimes E)\oplus D.
\]

\uline{N}ote that $\Lambda=U(sl_{4})(E_{1,2}-E_{2,1})$ and $\Lambda'=U(sl_{4})(E_{3,4}-E_{4,3})$
are highest weight modules with highest weight $\omega_{2}$. Define
$f:\Lambda'\rightarrow\Lambda$ by $f(x.(E_{3,4}-E_{4,3}))=x.(E_{1,2}-E_{2,1})$
for all $x\in U(sl_{4})$. This allows us to identify $\Lambda$ with
$\Lambda'$.

In (\ref{t2}) we list bases for all non-zero $\mathfrak{g}$-module
homomorphism spaces $\hom(X\otimes Y,Z)$ where $X,Y,Z\in\{\mathfrak{g},V,V',S,\Lambda,S',T\}$
and $X$ and $Y$ are both non-trivial. Note that all of them are
$1$-dimensional except the first one (which is 2-dimensional).

\begin{align}
\hom(\mathfrak{g}\otimes\mathfrak{g},\mathfrak{g}) & =span\{x\otimes y\mapsto xy-yx,\ x\otimes y\mapsto xy+yx-\frac{2}{4}\tra(xy)I\},\label{t2}\\
\hom(V\otimes V',\mathfrak{g}) & =span\{u\otimes v'\mapsto uv'^{t}-\frac{\tra(uv'^{t})}{4}I\},\nonumber \\
\hom(S\otimes\Lambda,\mathfrak{g}) & =span\{s\otimes\lambda\mapsto sf^{-1}(\lambda)\},\nonumber \\
\hom(S'\otimes\Lambda,\mathfrak{g}) & =span\{s'\otimes\lambda\mapsto s'\lambda\},\nonumber \\
\hom(\Lambda\otimes\Lambda,\mathfrak{g}) & =span\{\lambda\otimes\lambda'\mapsto\lambda f^{-1}(\lambda')-\frac{\tra(\lambda f^{-1}(\lambda'))}{4}I\},\nonumber \\
\hom(S\otimes S',\mathfrak{g}) & =span\{s\otimes s'\mapsto ss'-\frac{\tra(ss')}{4}I\},\nonumber \\
\hom(\mathfrak{g}\otimes V,V) & =span\{x\otimes v\mapsto xv\},\nonumber \\
\hom(\Lambda\otimes V',V) & =span\{\lambda\otimes v'\mapsto\lambda v'\},\nonumber \\
\hom(S\otimes V',V) & =span\{s\otimes v'\mapsto sv'\},\nonumber \\
\hom(\mathfrak{g}\otimes V',V') & =span\{x\otimes v'\mapsto x^{t}v'\},\nonumber \\
\hom(S'\otimes V,V') & =span\{s'\otimes v\mapsto s'v\},\nonumber \\
\hom(\Lambda\otimes V,V') & =span\{\lambda\otimes v\mapsto f^{-1}(\lambda)v\},\nonumber \\
\hom(\mathfrak{g}\otimes S,S) & =span\{x\otimes s\mapsto xs+sx^{t}\},\nonumber \\
\hom(V\otimes V,S) & =span\{u\otimes v\mapsto uv^{t}+vu^{t}\},\nonumber \\
\hom(\mathfrak{g}\otimes\Lambda,S) & =span\{x\otimes\lambda\mapsto x\lambda-\lambda x^{t}\}\nonumber \\
\hom(\mathfrak{g}\otimes\Lambda,S') & =span\{x\otimes\lambda\mapsto f^{-1}(\lambda)x-x^{t}f^{-1}(\lambda)\},\nonumber \\
\hom(S'\mathfrak{\otimes g},S') & =span\{s'\otimes x\mapsto s'x+x^{t}s'\},\nonumber \\
\hom(V'\otimes V',S') & =span\{u'\otimes v'\mapsto u'v'^{t}+v'u'^{t}\},\nonumber \\
\hom(\Lambda'\mathfrak{\otimes g},S') & =span\{\lambda'\otimes x\mapsto\lambda'x-x^{t}\lambda'\},\nonumber \\
\hom(\mathfrak{g}\otimes\Lambda,\Lambda) & =span\{x\otimes\lambda\mapsto x\lambda+\lambda x^{t}\},\nonumber \\
\hom(\mathfrak{g}\otimes S,\Lambda) & =span\{x\otimes s\mapsto xs-sx^{t}\},\nonumber \\
\hom(V\otimes V,\Lambda) & =span\{u\otimes v\mapsto uv^{t}-vu^{t}\},\nonumber \\
\hom(S'\mathfrak{\otimes g},\Lambda) & =span\{s'\otimes x\mapsto f(s'x-x^{t}s')\},\nonumber \\
\hom(V'\otimes V',\Lambda) & =span\{u'\otimes v'\mapsto f(u'v'^{t}-v'u'^{t})\},\nonumber \\
\hom(\mathfrak{g}\otimes\mathfrak{g},T) & =span\{x_{1}\otimes x_{2}\mapsto\frac{1}{4}\tra(x_{1}x_{2})\},\nonumber \\
\hom(V'\otimes V,T) & =span\{v^{t}\otimes u\mapsto\frac{1}{4}\tra(uv^{t})\},\nonumber \\
\hom(S\otimes S',T) & =span\{s\otimes s'\mapsto\frac{1}{4}\tra(ss')\},\nonumber \\
\hom(\Lambda\otimes\Lambda,T) & =span\{\lambda\otimes\lambda'\mapsto\frac{1}{4}\tra(\lambda f^{-1}(\lambda'))\}.\nonumber 
\end{align}

\label{mutiplication} Following the methods in \cite{key-1,key-10,key-7,key-6},
using the results of (\ref{t2}) and Table \ref{t13}, we may suppose
that the multiplication in $L$ is given as follows. For all $x,y\in sl_{4}$,
$u,v\in V$, $u',v'\in V'$, $s\in S$, $\lambda\in\Lambda$, $s'\in S'$,
$\lambda'\in\Lambda'$ and for all $a,a_{1},a_{2}\in A$, $b,b_{1},b_{2}\in B$,
$b',b_{1}',b_{2}'\in B'$, $c\in C$, $c'\in C'$, $e\in E$ and $d,d_{1},d_{2}\in D$, 

\begin{flalign}
 & [x\otimes a_{1},y\otimes a_{2}]=(x\circ y)\otimes\frac{[a_{1},a_{2}]}{2}+[x,y]\otimes\frac{a_{1}\circ a_{2}}{2}+(x\mid y)\langle a_{1},a_{2}\rangle,\label{main for}\\
 & [u\otimes b,v'\otimes b']=(uv'^{t}-\frac{\tra(uv'^{t})}{n}I)\otimes(b,b')_{A}+\frac{2}{n}\tra(uv'^{t})\langle b,b'\rangle=-[v'\otimes b',u\otimes b],\nonumber \\
 & [s\otimes c,s'\otimes c']=(ss'-(s\mid s')I)\otimes(c,c')_{A}+(s\mid s')\langle c,c'\rangle=-[s'\otimes c',s\otimes c],\nonumber \\
 & [\lambda_{1}\otimes e_{1},\lambda_{2}\otimes e_{2}]=(\lambda_{1}f^{-1}(\lambda_{2})-(\lambda_{1}\mid f^{-1}(\lambda_{2}))I)\otimes(e_{1},e_{2})_{A}+(\lambda_{1}\mid f^{-1}(\lambda_{2}))\langle e_{1},e_{2}\rangle,\nonumber \\
 & [u\otimes b_{1},v\otimes b_{2}]=(uv^{t}+vu^{t})\otimes\frac{(b_{1},b_{2})_{C}}{2}+(uv^{t}-vu^{t})\otimes\frac{(b_{1},b_{2})_{E}}{2},\nonumber \\
 & [u'\otimes b'_{1},v'\otimes b'_{2}]=(u'v'^{t}+v'u'^{t})\otimes\frac{(b'_{1},b'_{2})_{C'}}{2}+f(u'v'^{t}-v'u'^{t})\otimes\frac{(b'_{1},b'_{2})_{E}}{2},\nonumber \\
 & [x\otimes a,s\otimes c]=(xs+sx^{t})\otimes\frac{(a,c)_{C}}{2}+(xs-sx^{t})\otimes\frac{(a,c)_{E}}{2}=-[s\otimes c,x\otimes a],\nonumber \\
 & [x\otimes a,\lambda\otimes e]=(x\lambda+\lambda x^{t})\otimes\frac{(a,e)_{E}}{2}+(x\lambda-\lambda x^{t})\otimes\frac{(a,e)_{C}}{2},\nonumber \\
 & \,\,\,\,\,\,\,\,\,\,\,\,\,\,+(f^{-1}(\lambda)x-x^{t}f^{-1}(\lambda))\otimes\frac{(a,e)_{C'}}{2}=-[\lambda\otimes e,x\otimes a]\nonumber \\
 & [s'\otimes c',x\otimes a]=(s'x+x^{t}s')\otimes\frac{(c',a)_{C'}}{2}+f(s'x-x^{t}s')\otimes\frac{(c',a)_{E}}{2}=-[x\otimes a,s'\otimes c'],\nonumber \\
 & [s\otimes c,\lambda\otimes e]=sf^{-1}(\lambda)\otimes(c,e')_{A}=-[\lambda\otimes e,s\otimes c],\nonumber \\
 & [s'\otimes c',\lambda\otimes e]=s'\lambda\otimes(c',e)_{A}=-[\lambda\otimes e,s'\otimes c'],\nonumber \\
 & [x\otimes a,u\otimes b]=xu\otimes(a,b)_{B}=-[u\otimes b,x\otimes a],\nonumber \\
 & [s'\otimes c',u\otimes b]=s'u\otimes(c',b)_{B'}=-[u\otimes b,s'\otimes c'],\nonumber \\
 & [\lambda\otimes e,u\otimes b]=f^{-1}(\lambda)u\otimes(e,b)_{B'}=-[u\otimes b,\lambda\otimes e],\nonumber \\
 & [u'\otimes b',x\otimes a]=x^{t}u'\otimes(b',a)_{B'}=-[x\otimes a,u'\otimes b'],\nonumber \\
 & [u'\otimes b',s\otimes c]=su'\otimes(b',c)_{B}=-[s\otimes c,u'\otimes b'],\nonumber \\
 & [u'\otimes b',\lambda\otimes e]=-\lambda u'\otimes(b',e)_{B}=-[\lambda\otimes e,u'\otimes b'],\nonumber \\
 & [d,x\otimes a]=x\otimes da=-[x\otimes a,d],\nonumber \\
 & [d,u\otimes b]=u\otimes db=-[u\otimes b,d],\nonumber \\
 & [d,s\otimes c]=s\otimes dc=-[s\otimes c,d],\nonumber \\
 & [d,\lambda\otimes e]=\lambda\otimes de=-[\lambda\otimes e,d],\nonumber \\
 & [d,s'\otimes c']=s'\otimes dc'=-[s'\otimes c',d],\nonumber \\
 & [d,u'\otimes b']=u'\otimes db'=-[u'\otimes b',d],\nonumber \\
 & [d_{1},d_{2}]\in D,\nonumber 
\end{flalign}
All other products of the homogeneous components of the decomposition
(\ref{eq:drezh mine}) are zero.

\section{\label{4} Coordinate algebra of $(\Theta_{n},sl_{n})$-graded Lie
algebras $(n=3,4)$}

Let $L$ be a $\Theta_{n}$-graded Lie algebra with the grading subalgebra
$\mathfrak{g}\cong sl_{n}$. Throughout this section we assume that
$n=4$ or $n=3$ and the conditions (\ref{condition}) hold. Let $\mathfrak{g}^{\pm}=\{x\in sl_{n}\mid x^{t}=\pm x\}$.
Then 
\begin{equation}
\mathfrak{g}\otimes A=(\mathfrak{g^{+}}\oplus\mathfrak{g^{-}})\otimes A=(\mathfrak{g^{+}}\otimes A)\oplus(\mathfrak{g^{-}}\otimes A)=(\mathfrak{g^{+}}\otimes A^{-})\oplus(\mathfrak{g}^{-}\otimes A^{+})\label{AA-1}
\end{equation}
where $A^{\pm}$ is a copy of the vector space $A$. Recall that we
identify $\mathfrak{g}$ with $\mathfrak{g}\otimes1$ where $1$ is
a distinguished element of $A$. We denote by $a^{\pm}$ the image
of $a\in A$ in the space $A^{\pm}$. Denote
\begin{align*}
\mathfrak{a} & :=A^{+}\oplus A^{-}\oplus C\oplus C'\oplus E\oplus E'\text{ for }n=3.\\
\mathfrak{a} & :=A^{+}\oplus A^{-}\oplus C\oplus E\oplus C'\text{\quad and\quad}\mathfrak{b}:=\mathfrak{a}\oplus B\oplus B'\text{ for }n=4.
\end{align*}

In Section \ref{3} we described the multiplicative structures of
these Lie algebras. In this section we describe the coordinate algebras
of them. and we show that the product in $L$ induces an algebra structure
on both $\mathfrak{a}$ and $\mathfrak{b}$. 

\subsection{\label{sec:Unital-associative-algebra aa-1}Unital algebra $\mathfrak{a}$}

\label{fomulas-1}We are going to define Lie and Jordan multiplication
on $\mathfrak{a}$ by extending the bilinear products given in Tables
\ref{t4-1} and \ref{t4-1-2} in a natural way. It can be shown that
all products $(\alpha_{1},\alpha_{2})_{Z}$ with $\alpha_{1},\alpha_{2}\in\mathfrak{a}$
are either symmetric or skew-symmetric, except in the cases ($\alpha_{1}=c$
and $\alpha_{2}=e$ or $\alpha_{1}=c'$ and $\alpha_{2}=e'$) for
$n=3$). This is why we will write $(\alpha_{1}\circ\alpha_{2})_{Z}$
or $[\alpha_{1},\alpha_{2}]_{Z}$, respectively, instead of $(\alpha_{1},\alpha_{2})_{Z}$.
The aim of this subsection is to show that $\mathfrak{a}$ is a unital
algebra with respect to the new multiplication given by $\alpha_{1}\alpha_{2}:=\frac{[\alpha_{1},\alpha_{2}]}{2}+\frac{\alpha_{1}\circ\alpha_{2}}{2}$
for all homogeneous $\alpha_{1},\alpha_{2}\in\mathfrak{a}$ with the
products $[\ ]$ and $\circ$ given by Table \ref{t4-1}, except in
the cases ($\alpha_{1}=c$ and $\alpha_{2}=e$ or $\alpha_{1}=c'$
and $\alpha_{2}=e'$) for $n=3$ and $a^{\pm}\lambda$ for $n=4$,
see (\ref{eq:except hom}).
\begin{rem}
\label{FF-1} By using the same arguments of \cite[Remark 4.1.1]{key-9})
some of the products in (\ref{main for}) and (\ref{t2-1}) can be
rewritten in terms of symmetric and skew-symmetric elements. For example
for $n=4$ these products can be written as follows: For all $x^{\pm},x_{1}^{\pm},x_{2}^{\pm}\in\mathfrak{g^{\pm}}$,
$u,v\in V$, $u',v'\in V'$, $s\in S$, $\lambda,\lambda_{1},\lambda_{2}\in\Lambda$,
$s'\in S'$ and for all $a^{\pm},a_{1}^{\pm},a_{2}^{\pm}\in A$, $b,b_{1},b_{2}\in B$,
$b',b_{1}',b_{2}'\in B'$, $c\in C$, $c'\in C'$, $e\in E$,  $d\in D$,
\begin{align*}
[x_{1}^{+}\otimes a_{1}^{-},x_{2}^{+}\otimes a_{2}^{-}] & =x_{1}^{+}\circ x_{2}^{+}\otimes\frac{[a_{1}^{-},a_{2}^{-}]_{A^{-}}}{2}+[x_{1}^{+},x_{2}^{+}]\otimes\frac{(a_{1}^{-}\circ a_{2}^{-})_{A^{+}}}{2}+(x_{1}^{+}\mid x_{2}^{+})\langle a_{1}^{-},a_{2}^{-}\rangle\text{,}\\{}
[x_{1}^{-}\otimes a_{1}^{+},x_{2}^{-}\otimes a_{2}^{+}] & =x_{1}^{-}\circ x_{2}^{-}\otimes\frac{[a_{1}^{+},a_{2}^{+}]_{A^{-}}}{2}+[x_{1}^{-},x_{2}^{-}]\otimes\frac{(a_{1}^{+}\circ a_{2}^{+})_{A^{+}}}{2}+(x_{1}^{-}\mid x_{2}^{-})\langle a_{1}^{+},a_{2}^{+}\rangle\text{,}\\{}
[x_{1}^{+}\otimes a_{1}^{-},x_{1}^{-}\otimes a_{1}^{+}] & =x_{1}^{+}\diamond x_{1}^{-}\otimes\frac{[a_{1}^{-},a_{1}^{+}]_{A^{+}}}{2}+[x_{1}^{+},x_{1}^{-}]\otimes\frac{(a_{1}^{-}\circ a_{1}^{+})_{A^{-}}}{2}\text{.}\\{}
[x^{+}\otimes a^{-},s\otimes c] & =x^{+}\diamond s\otimes\frac{[a^{-},c]_{C}}{2}+[x^{+},s]\otimes\frac{(a^{-}\circ c)_{E}}{2},\\{}
[x^{-}\otimes a^{+},s\otimes c] & =x^{-}\diamond s\otimes\frac{[a^{+},c]_{E}}{2}+[x^{-},s]\otimes\frac{(a^{+}\circ c)_{C}}{2}.\\{}
[x^{+}\otimes a^{-},\lambda\otimes e] & =x^{+}\diamond\lambda\otimes\frac{[a^{-},e]_{E}}{2}+[x^{+},\lambda]\otimes\frac{(a^{-}\circ e)_{C}}{2}+[x^{+},f^{-1}(\lambda)]\otimes\frac{(a^{-}\circ e)_{C'}}{2}.\\{}
[x^{-}\otimes a^{+},\lambda\otimes e] & =x^{-}\diamond\lambda\otimes\frac{[a^{+},e]_{C}}{2}+[x^{-},\lambda]\otimes\frac{(a^{+}\circ e)_{E}}{2}+x^{-}\diamond f^{-1}(\lambda)\otimes\frac{[a^{+},e]_{C'}}{2}.\\{}
[s'\otimes c',x^{+}\otimes a^{-}] & =s'\diamond x^{+}\otimes\frac{[c',a^{-}]_{C'}}{2}+f([s',x^{+}])\otimes\frac{(c'\circ a^{-})_{E}}{2},\\{}
[s'\otimes c',x^{-}\otimes a^{+}] & =f(s'\diamond x^{-})\otimes\frac{[c',a^{+}]_{E}}{2}+[s',x^{-}]\otimes\frac{(c'\circ a^{+})_{C'}}{2}.\\{}
[s\otimes c,s'\otimes c'] & =s\circ s'\otimes\frac{[c,c']_{A^{-}}}{2}+[s,s']\otimes\frac{(c\circ c')_{A^{+}}}{2}+(s\mid s')\langle c,c'\rangle.\\{}
[\lambda_{1}\otimes e_{1},\lambda_{2}\otimes e_{2}] & =\lambda_{1}\circ f^{-1}(\lambda_{2})\otimes\frac{[e_{1},e_{2}]_{A^{-}}}{2}+[\lambda_{1},f^{-1}(\lambda_{2})]\otimes\frac{(e_{1}\circ e_{2})_{A^{+}}}{2}+(\lambda_{1}\mid f^{-1}(\lambda_{2}))\langle e_{1},e_{1}\rangle.\\{}
[s\otimes c,\lambda\otimes e] & =s\diamond\lambda\otimes\frac{[c,e]_{A^{+}}}{2}+[s,\lambda]\otimes\frac{(c\circ e)_{A^{-}}}{2}.\\{}
[s'\otimes c',\lambda\otimes e] & =s'\diamond\lambda\otimes\frac{[c',e]_{A^{+}}}{2}+[s',\lambda]\otimes\frac{(c'\circ e)_{A^{-}}}{2}.
\end{align*}
\end{rem}

The mappings $\alpha\otimes\beta\mapsto(\alpha\circ\beta)_{Z_{1}}$
and $\alpha\otimes\beta\mapsto[\alpha,\beta]_{Z_{2}}$ can be extended
to $Y\otimes X$ in a consistent way by defining $(\beta\circ\alpha)_{Z_{1}}=(\alpha\circ\beta)_{Z_{1}}$
and $[\beta,\alpha]_{Z_{2}}=-[\alpha,\beta]_{Z_{2}}$ . For $n=3$,
$\alpha\in C$ and $\beta\in E$ we will write $\alpha\beta$ (resp.
$\beta\alpha$) instead of $(\alpha,\beta)_{Z}$ (resp. $(\beta,\alpha)_{Z}$).
and the map can be extended by by defining $(\beta\alpha)_{E'}=-(\alpha\beta)_{E'}$(
resp,~$(\beta\alpha)_{E}=-(\alpha\beta)_{E}$). In Tables \ref{t4-1}
and \ref{t4-1-2} below, if the cell in row $X$ and column $Y$ contains
$(Z_{1},\circ)$, and $(Z_{2},[\ ])$ this means that there is a symmetric
bilinear map $X\times Y\rightarrow Z_{1},$ given by $\alpha\otimes\beta\mapsto(\alpha\circ\beta)_{Z_{1}}$
and a skew symmetric bilinear map $X\times Y\rightarrow Z_{2},$ given
by $\alpha\otimes\beta\mapsto[\alpha,\beta]_{Z_{2}}$ $(\alpha\in X,\beta\in Y)$. 

\begin{table}[H]
\begin{tabular}{|c|r@{\extracolsep{0pt}.}l|c|c|c|c|c|}
\hline 
$.$ & \multicolumn{2}{c|}{$A^{+}$} & $A^{-}$ & $C$ & $C'$ & $E$ & $E'$\tabularnewline
\hline 
\hline 
$A^{+}$ & \multicolumn{2}{c|}{$\begin{array}{c}
(A^{+},\circ)\\
(A^{-},[\ ])
\end{array}$} & $\begin{array}{c}
(A^{-},\circ)\\
(A^{+},[\ ])
\end{array}$ & $\begin{array}{c}
(C,\circ)\end{array}$ & $\begin{array}{c}
(C',\circ)\end{array}$ & $\begin{array}{c}
(C,\circ)\\
(E,[\ ])
\end{array}$ & $\begin{array}{c}
(E',\circ)\\
(C',[\ ])
\end{array}$\tabularnewline
\hline 
$A^{-}$ & \multicolumn{2}{c|}{$\begin{array}{c}
(A^{-},\circ)\\
(A^{+},[\ ])
\end{array}$} & $\begin{array}{c}
(A^{+},\circ)\\
(A^{-},[\ ])\text{ }
\end{array}$ & $\begin{array}{c}
(C,[\ ])\end{array}$ & $\begin{array}{c}
(C',[\ ])\end{array}$ & $\begin{array}{c}
(E,\circ)\\
(C,[\ ])
\end{array}$ & $\begin{array}{c}
(C',\circ)\\
(E',[\ ])
\end{array}$\tabularnewline
\hline 
$C$ & \multicolumn{2}{c|}{$\begin{array}{c}
(C,\circ)\end{array}$} & $\begin{array}{c}
(C,[\ ])\text{ }\end{array}$ & $0$ & $\begin{array}{c}
(A^{+},\circ)\\
(A^{-},[\ ])
\end{array}$ & $\begin{array}{c}
E'\end{array}$ & $\begin{array}{c}
(A^{-},\circ)\\
(A^{+},[\ ])
\end{array}$\tabularnewline
\hline 
$C'$ & \multicolumn{2}{c|}{$\begin{array}{c}
(C',\circ)\end{array}$} & $\begin{array}{c}
(C',[\ ])\end{array}$ & $\begin{array}{c}
(A^{+},\circ)\\
(A^{-},[\ ])
\end{array}$ & $0$ & $\begin{array}{c}
(A^{-},\circ)\\
(A^{+},[\ ])
\end{array}$ & $\begin{array}{c}
E\end{array}$\tabularnewline
\hline 
$E$ & \multicolumn{2}{c|}{$\begin{array}{c}
(C,\circ)\\
(E,[\ ])
\end{array}$} & $\begin{array}{c}
(E,\circ)\\
(C,[\ ])
\end{array}$ & $\begin{array}{c}
E'\end{array}$ & $\begin{array}{c}
(A^{-},\circ)\\
(A^{+},[\ ])
\end{array}$ & $\begin{array}{c}
(E',\circ)\mbox{ }\\
(C',[\ ])
\end{array}$ & $\begin{array}{c}
(A^{+},\circ)\mbox{ }\\
(A^{-},[\ ])
\end{array}$\tabularnewline
\hline 
$E''$ & \multicolumn{2}{c|}{$\begin{array}{c}
(E',\circ)\\
(C',[\ ])
\end{array}$} & $\begin{array}{c}
(C',\circ)\\
(E',[\ ])
\end{array}$ & $\begin{array}{c}
(A^{-},\circ)\\
(A^{+},[\ ])
\end{array}$ & $\begin{array}{c}
E\end{array}$ & $\begin{array}{c}
(A^{+},\circ)\mbox{ }\\
(A^{-},[\ ])
\end{array}$ & $\begin{array}{c}
(E,\circ)\\
(C,[\ ])
\end{array}$\tabularnewline
\hline 
\end{tabular}

\caption{Products of homogeneous components of $\mathfrak{a}$ for $n=3$}

\label{t4-1}
\end{table}

\begin{table}[H]
\begin{tabular}{|c|r@{\extracolsep{0pt}.}l|c|c|c|c|}
\hline 
$.$ & \multicolumn{2}{c|}{$A^{+}$} & $A^{-}$ & $C$ & $E$ & $C'$\tabularnewline
\hline 
\hline 
$A^{+}$ & \multicolumn{2}{c|}{$\begin{array}{c}
(A^{+},\circ)\\
(A^{-},[\ ])
\end{array}$} & $\begin{array}{c}
(A^{-},\circ)\\
(A^{+},[\ ])
\end{array}$ & $\begin{array}{c}
(C,\circ)\\
(E,[\ ])
\end{array}$ & $\begin{array}{c}
(E,\circ)\\
(C,[\ ])\\
(C',\circ)
\end{array}$ & $\begin{array}{c}
(C',\circ)\\
(E,[\ ])
\end{array}$\tabularnewline
\hline 
$A^{-}$ & \multicolumn{2}{c|}{$\begin{array}{c}
(A^{-},\circ)\\
(A^{+},[\ ])
\end{array}$} & $\begin{array}{c}
(A^{+},\circ)\\
(A^{-},[\ ])\text{ }
\end{array}$ & $\begin{array}{c}
(E,\circ)\\
(C,[\ ])
\end{array}$ & $\begin{array}{c}
(C,\circ)\\
(E,[\ ])\\
(C',\circ)
\end{array}$ & $\begin{array}{c}
(E,\circ)\\
(C',[\ ])
\end{array}$\tabularnewline
\hline 
$C$ & \multicolumn{2}{c|}{$\begin{array}{c}
(C,\circ)\\
(E,[\ ])
\end{array}$} & $\begin{array}{c}
(E,\circ)\mbox{ }\\
(C,[\ ])\text{ }
\end{array}$ & $0$ & $\begin{array}{c}
(A^{-},\circ)\\
(A^{+},[\ ])
\end{array}$ & $\begin{array}{c}
(A^{+},\circ)\\
(A^{-},[\ ])
\end{array}$\tabularnewline
\hline 
$E$ & \multicolumn{2}{c|}{$\begin{array}{c}
(E,\circ)\\
(C,[\ ])\\
(C',\circ)
\end{array}$} & $\begin{array}{c}
(C,\circ)\\
(E,[\ ])\\
(C',\circ)
\end{array}$ & $\begin{array}{c}
(A^{-},\circ)\\
(A^{+},[\ ])
\end{array}$ & $\begin{array}{c}
(A^{+},\circ)\\
(A^{-},[\ ])
\end{array}$ & $\begin{array}{c}
(A^{-},\circ)\\
(A^{+},[\ ])
\end{array}$\tabularnewline
\hline 
$C'$ & \multicolumn{2}{c|}{$\begin{array}{c}
(C',\circ)\\
(E,[\ ])
\end{array}$} & $\begin{array}{c}
(E,\circ)\\
(C',[\ ])
\end{array}$ & $\begin{array}{c}
(A^{+},\circ)\\
(A^{-},[\ ])
\end{array}$ & $\begin{array}{c}
(A^{-},\circ)\\
(A^{+},[\ ])
\end{array}$ & $0$\tabularnewline
\hline 
\end{tabular}

\caption{Products of homogeneous components of $\mathfrak{a}$ for $n=4$}

\label{t4-1-2}
\end{table}

We are going to show that $\mathfrak{a}=A^{+}\oplus A^{-}\oplus C\oplus E\oplus C'$
is an associative algebra with respect to multiplication defined as
follows: 
\[
\alpha_{1}\alpha_{2}:=\frac{[\alpha_{1},\alpha_{2}]}{2}+\frac{\alpha_{1}\circ\alpha_{2}}{2}
\]
 for all homogeneous $\alpha_{1},\alpha_{2}\in\mathfrak{a}$ with
the products $[\ ]$ and $\circ$ given by Table \ref{t4-1}, except
in the case $a^{\pm}\lambda$ which is equal to 
\begin{align}
a^{-}\lambda & =\frac{[a^{-},e]_{E}}{2}+\frac{(a^{-}\circ e)_{C}}{2}+\frac{(a^{-}\circ e)_{C'}}{2},\label{eq:except hom}\\
a^{+}\lambda & =\frac{[a^{+},e]_{C}}{2}+\frac{(a^{+}\circ e)_{E}}{2}+\frac{[a^{+},e]_{C'}}{2}.\nonumber 
\end{align}
 Note that $[\alpha_{1},\alpha_{2}]=\alpha_{1}\alpha_{2}-\alpha_{2}\alpha_{1}$
and $\alpha_{1}\circ\alpha_{2}=\alpha_{1}\alpha_{2}+\alpha_{2}\alpha_{1}$.
\begin{prop}
1)\label{cor A is sub-1} $\mathcal{A}=A^{-}\oplus A^{+}$ is an associative
subalgebra of $\mathfrak{a}$ with identity element $1^{+}$.

2) $C\oplus E$ and $C'\oplus E'$ are $\mathcal{A}$-bimodules for
$n=3$.

3) $C\oplus E\oplus C'$ are $\mathcal{A}$-bimodule for $n=4$.
\end{prop}

\begin{proof}
1) This is similar to the proof of\cite[Theorem 4.1.3]{key-9}.

(2) and (3) can be deduce from tensor product decompositions for $sl_{n}$
$(n=3,4)$, see Tables \ref{t12-2} and \ref{t13}.
\end{proof}
\begin{thm}
\label{a invol-1} The linear transformation $\gamma:\mathfrak{a}\rightarrow\mathfrak{a}$
defined by
\[
\gamma(a^{-})=-a^{-},\gamma(a^{+})=a^{+},\gamma(c)=-c,\gamma(e)=e,\gamma(c')=-c',
\]
is an antiautomorphism of order 2 of the algebra $\mathfrak{a}.$
\end{thm}

\begin{proof}
See \cite[Theorem 4.1.6]{key-10}.
\end{proof}

\subsection{\label{sec:Coordinate-algebra bb-1}Coordinate algebra $\mathfrak{b}$ }

\label{six-1} Let $L$ be an $\Theta_{4}$-graded Lie algebra with
the grading subalgebra $\mathfrak{g}\cong sl_{4}$. Recall that we
denote
\begin{align*}
\mathfrak{a} & :=A^{+}\oplus A^{-}\oplus C\oplus E\oplus C'\text{\quad and\quad}\mathfrak{b}:=\mathfrak{a}\oplus B\oplus B.
\end{align*}
The aim of this subsection is to show that $\mathfrak{b}$ is an algebra
with identity $1^{+}$ with respect to the multiplication extending
that on $\mathfrak{a}$ given in Table \ref{t5-1}. It can be shown
that all products $(\beta_{1},\beta_{2})_{Z}$ with $\beta_{1},\beta_{2}\in B\oplus B'$
are either symmetric or skew-symmetric. This is why we will write
$(\beta_{1}\circ\beta_{2})_{Z}$ or $[\beta_{1},\beta_{2}]_{Z}$,
respectively, instead of $(\beta_{1},\beta_{2})_{Z}$. For $\alpha\in\mathfrak{a}$
and $\beta\in B\oplus B'$ we will write $\alpha\beta$ (resp. $\beta\alpha$)
instead of $(\alpha,\beta)_{Z}$ (resp. $(\beta,\alpha)_{Z}$). Let
$b\in B$ and $b'\in B$. We define $b\alpha:=\gamma(\alpha)b$ and
$\alpha b':=b'\gamma(\alpha)$. We will show that $B\oplus B'$ is
an  $\mathfrak{a}$-bimodule. 

Recall that 
\[
x\otimes a=\frac{(x+x^{t})}{2}\otimes a+\frac{(x-x^{t})}{2}\otimes a\in\mathfrak{g^{+}}\otimes A+\mathfrak{g^{-}}\otimes A.
\]

Let $u\otimes b\in V\otimes B$ and $v'\otimes b'\in V'\otimes B'$.
We need the following formula from (\ref{main for}): 
\[
[u\otimes b,v'\otimes b']=(uv'^{t}-\frac{\tra(uv'^{t})}{4}I)\otimes(b,b')_{A}+\frac{2\tra(uv'^{t})}{4}\langle b,b'\rangle.
\]
By splitting $(b,b')_{A}$ into symmetric and skew-symmetric parts
and using the equations 
\begin{align*}
(uv'^{t}-\frac{\tra(uv'^{t})}{4}I)+(uv'^{t}-\frac{\tra(uv'^{t})}{4}I)^{t} & =uv'^{t}+v'u^{t}-\frac{2\tra(uv'^{t})}{4}I,\\
(uv'^{t}-\frac{\tra(uv'^{t})}{4}I)-(uv'^{t}-\frac{\tra(uv'^{t})}{4}I)^{t} & =uv'^{t}-v'u^{t},
\end{align*}
we get 
\begin{align}
[u\otimes b,v'\otimes b'] & =(uv'^{t}+v'u^{t}-\frac{2\tra(uv'^{t})}{4}I)\otimes\frac{[b,b']_{A^{-}}}{2}+\nonumber \\
 & \quad(uv'^{t}-v'u^{t})\otimes\frac{(b\circ b')_{A^{+}}}{2}+\frac{2\tra(uv'^{t})}{4}\langle b,b'\rangle.\label{eq:bb' formull-1}
\end{align}

Let $b,b_{1},b_{2}\in B$ and $b',b_{1}',b_{2}'\in B'$. Using (\ref{main for})
and (\ref{eq:bb' formull-1}) we get 
\begin{align}
[u\otimes b_{1},v\otimes b_{2}] & =(uv^{t}+vu^{t})\otimes\frac{[b_{1},b_{2}]_{C}}{2}+(uv^{t}-vu^{t})\otimes\frac{(b_{1}\circ b_{2})_{E}}{2},\nonumber \\{}
[u'\otimes b'_{1},v'\otimes b'_{2}] & =(u'v'^{t}+v'u'^{t})\otimes\frac{[b'_{1},b'_{2}]_{C'}}{2}+f(u'v'^{t}-v'u'^{t})\otimes\frac{(b'_{1}\circ b'_{2})_{E}}{2},\nonumber \\{}
[u\otimes b,v'\otimes b'] & =(uv'^{t}+v'u^{t}-\frac{2\tra(uv'^{t})}{4}I)\otimes\frac{[b,b']_{A^{-}}}{2}+\nonumber \\
 & \quad(uv'^{t}-v'u^{t})\otimes\frac{(b\circ b')_{A^{+}}}{2}+\frac{2\tra(uv'^{t})}{4}\langle b,b'\rangle.\label{eq:formulla for natural elements-1}
\end{align}
We define 
\begin{eqnarray*}
b_{1}b_{2}:=\frac{[b_{1},b_{2}]_{C}}{2}+\frac{(b_{1}\circ b_{2})_{E}}{2}, & \  & b'_{1}b'_{2}:=\frac{[b'_{1},b'_{2}]_{C'}}{2}+\frac{(b'_{1}\circ b'_{2})_{E}}{2},\\
bb':=\frac{[b,b']_{A^{-}}}{2}+\frac{(b\circ b')_{A^{+}}}{2}, & \  & b'b:=-\frac{[b,b']_{A^{-}}}{2}+\frac{(b\circ b')_{A^{+}}}{2}.
\end{eqnarray*}
 Then $\mathfrak{b}=\mathfrak{a}\oplus B\oplus B'$ is an algebra
with multiplication extending that on $\mathfrak{a.}$ The following
table describes the products of homogeneous elements of $\mathfrak{b}$
(use Table \ref{t4-1} for the products on $\mathfrak{a}$). 

\begin{table}[H]
\begin{tabular}{|c|c|c|c|c|c|}
\hline 
. & $A^{+}+A^{-}$ & $C+E$ & $C'$ & $B$ & $B'$\tabularnewline
\hline 
\hline 
$A^{+}+A^{-}$ & $A^{+}+A^{-}$ & $C+E+C'$ & $C'+E$ & $B$ & $B'$\tabularnewline
\hline 
$C+E$ & $C+E+C'$ & $0$ & $A^{+}+A^{-}$ & $0$ & $B$\tabularnewline
\hline 
$C'$ & $C'+E$ & $A^{+}+A^{-}$ & $0$ & $B$' & 0\tabularnewline
\hline 
$B$ & $B$ & $0$ & $B'$ & $C+E$ & $A^{+}+A^{-}$\tabularnewline
\hline 
$B'$ & $B'$ & $B$ & $0$ & $A^{+}+A^{-}$ & $C'+E$\tabularnewline
\hline 
\end{tabular}

\caption{Products in $\mathfrak{b}$}

\label{t5-1}
\end{table}

Using the same arguments of \cite[Section 4.2]{key-9} we get the
following properties:
\begin{thm}
\label{b involution-1} (1) The linear transformation $\eta:\mathfrak{b}\rightarrow\mathfrak{b}$
defined by $\eta(\alpha)=\gamma(\alpha)$, $\eta(b)=b$ and $\eta(b')=b'$
for all $\alpha\in\mathfrak{a}$, $b\in B$ and $b'\in B'$ is an
antiautomorphism of order $2$ of the algebra $\mathfrak{b}$.

(2) $1^{+}$ is the identity element of $\mathfrak{b}$. 

(3) Let $b\in B$, $b'\in B'$ and $\alpha\in\mathfrak{a}$ .Then
\begin{align*}
[z\otimes\alpha,u\otimes b] & =zu\otimes\alpha b=-[u\otimes b,z\otimes\alpha],\\{}
[u'\otimes b',z\otimes\alpha] & =z^{t}u'\otimes b'\alpha=-[z\otimes\alpha,u'\otimes b'].
\end{align*}

(4) $B\oplus B'$ is an  $\mathfrak{a}$-bimodule.

(5) $B$ and $B'$ are $\mathcal{\mathcal{A}}$-bimodules.
\end{thm}

\begin{proof}
(1) This is similar to \cite[Theorem 4.2.1]{key-9}.

(2) This is similar to \cite[Theorem 4.2.2]{key-9}.

(3) We deduce this by using (\ref{main for}) and Table \ref{t5-1}. 

(4) See \cite[Proposition 4.2.4]{key-9}. 

(5) We deduce this from the properties that $B$ and $B'$ are invariant
under multiplication by $\mathcal{\mathcal{A}}$, see Table \ref{t5-1}. 
\end{proof}
The mapping $\langle,\rangle:X\otimes X'\rightarrow D$ with $X=B,C$
can be extended to $X'\otimes X$ in a consistent way by defining
$\langle x',x\rangle:=-\langle x,x'\rangle$. Let $X,Y\in\{A^{+},A^{-},B,B',C,C',E\}$.
Recall also the maps $\langle,\rangle:A^{\pm}\otimes A^{\pm}\rightarrow D$
described previously (see Remark \ref{FF-1}(a)). For the convenience,
we extend the mappings to the whole space $\mathfrak{b}$ by defining
the remaining $\langle X,Y\rangle$ to be zero. Hence 
\[
\langle\mathfrak{b},\mathfrak{b}\rangle=\langle A^{+},A^{+}\rangle+\langle A^{-},A^{-}\rangle+\langle B,B'\rangle+\langle C,C'\rangle+\langle E,E\rangle.
\]
It follows from condition $(\Gamma3)$ in Definition \ref{def of gamma}
that 
\begin{equation}
D=\langle\mathfrak{b},\mathfrak{b}\rangle.\label{eq:kk-1}
\end{equation}

\begin{prop}
\label{derivation 2-1}(1) $[d,\langle\alpha,\beta\rangle]=\langle d\alpha,\beta\rangle+\langle\alpha,d\beta\rangle$
for all $\alpha,\beta\in\ensuremath{\mathfrak{b}}$ and $d\in D.$

(2) $\langle A^{+},A^{+}\rangle$, $\langle A^{-},A^{-}\rangle$,
$\langle B,B'\rangle$, $\langle C,C'\rangle$ and $\langle E,E\rangle$
are ideals of the Lie algebra $D$.

(3) $D$ acts by derivations on $\mathfrak{b}$ and leaves all subspaces
$A^{+},A^{-},B,\dots,E$ invariant. 
\end{prop}

\begin{proof}
This is similar to \cite[Proposition 4.2.8]{key-9}.
\end{proof}

\appendix
\printindex
\end{document}